\documentclass[12pt]{amsart}

\usepackage[latin1]{inputenc}
\usepackage[english]{babel}
\usepackage{amsthm}
\usepackage{amsmath}
\usepackage{amsxtra}
\usepackage{mathrsfs}
\usepackage{amssymb}
\usepackage{bbm}
\usepackage[dvips]{graphicx}
\usepackage{graphicx}

\input xy
\xyoption{all}

\newtheorem{teo}{Theorem}[section]
\newtheorem{lemma}[teo]{Lemma}

\newtheorem{defi}[teo]{Definition}
\newtheorem{coro}[teo]{Corollary}

\newtheorem{cla}[teo]{Claim}
\newtheorem*{subclaim*}{Subclaim}
\newtheorem{prop}[teo]{Proposition}

\theoremstyle{remark}
\newtheorem{remark}[teo]{Remark}

\begin{document}

\newcommand{\ran}{{\rm ran}}
\newcommand{\cof}{{\rm cof}}
\newcommand{\dom}{{\rm dom}}
\newcommand{\I}{I}
\newcommand{\N}{\mathcal{N}}
\newcommand{\up}{\upharpoonright}
\newcommand{\rest}{\! \upharpoonright \!}
\newcommand{\urltilde}{\kern -.15em\lower .7ex\hbox{~}\kern .04em}

\newcommand\M{\mathcal{M}}
\newcommand\Acal{\mathscr{A}}
\newcommand\Bcal{\mathscr{B}}
\newcommand\Ccal{\mathscr{C}}
\newcommand\Dcal{\mathscr{D}}
\newcommand\Ecal{\mathscr{E}}
\newcommand\Fcal{\mathscr{F}}
\newcommand\Hcal{\mathscr{H}}
\newcommand\Ical{\mathscr{I}}
\newcommand\Ncal{\mathscr{N}}
\newcommand\Mcal{\mathscr{M}}
\newcommand\Pcal{\mathscr{P}}
\newcommand\Qcal{\mathscr{Q}}
\newcommand\calQ{\mathcal{Q}}
\newcommand\Rcal{\mathscr{R}}
\newcommand\Scal{\mathscr{S}}
\newcommand\Tcal{\mathscr{T}}
\newcommand\Ucal{\mathcal{U}}
\newcommand\Vcal{\mathscr{V}}
\newcommand\Zcal{\mathscr{Z}}
\newcommand\Rbb{\mathbb{R}}
\newcommand\Nfrak{\mathfrak{N}}
\newcommand\Pfrak{\mathfrak{P}}
\newcommand\restrict{\restriction}

\newcommand{\ORD}{{\rm ORD}}

\newcommand\Gdot{\dot{G}}
\newcommand\Qdot{\dot{Q}}

\newcommand{\diff}{\operatorname{\mathrm diff}}
\newcommand{\Ht}{\operatorname{\mathrm ht}}
\newcommand{\lev}{\operatorname{\mathrm lev}}
\newcommand{\meet}{\wedge}
\newcommand\triord{\triangleleft}
\newcommand{\Th}{{}^{\mathrm{th}}}
\newcommand{\St}{{}^{\mathrm{st}}}
\newcommand\axiom{\mathrm}
\newcommand\MA{\axiom{MA}}
\newcommand\MM{\axiom{MM}}
\newcommand\PFA{\axiom{PFA}}
\newcommand\BPFA{\axiom{BPFA}}
\newcommand\MRP{\axiom{MRP}}
\newcommand\SRP{\axiom{SRP}}
\newcommand\ZFC{\axiom{ZFC}}
\newcommand\CAT{\axiom{CAT}}
\newcommand\CH{\axiom{CH}}
\newcommand{\<}{\langle}
\renewcommand{\>}{\rangle}
\newcommand\mand{\textrm{ and }}
\renewcommand{\diamond}{\diamondsuit}
\newcommand{\forces}{\Vdash}

\newcommand{\cf}{{\mbox{cof}}}
\newcommand{\height}{\Ht}
\newcommand\NS{\mathrm{NS}}

\title[Iteration of semiproper forcing revisited]{Iteration of semiproper forcing revisited}

\author[Boban Veli\v{c}kovi\'{c}]{Boban Veli\v{c}kovi\'{c}}
%\revauthor{Veli\v{c}kovi\'c, Boban}
\email{boban@imj-prg.fr}
\urladdr{http://www.logique.jussieu.fr/\urltilde boban}
\address{Institut de Math\'ematiques de Jussieu, IMJ-PRG,
Universit\'e Paris Diderot,
75205 Paris Cedex 13
France}


\maketitle

\begin{abstract}{ We present a method for iterating semiproper forcing
which uses side conditions and is inspired by the technique recently
introduced by Neeman.}
\end{abstract}

\section*{Introduction}

One of the most important issues in the theory of forcing is whether certain classes
of nicely behaved forcing notions can be iterated while remaining in the same class.
We are primarily interested in preserving certain cardinals, but  often
we also wish to preserve some particular objects or properties of the ground model.
The theory of iterated forcing was developed largely by Shelah (see \cite{Sh_P}),
who isolated important classes of posets such as proper and semiproper
forcing and developed suitable iteration techniques for them.
In order to iterate proper forcing Shelah uses {\em countable support} iteration,
initially developed by Laver \cite{LaverBC}, Baumgartner \cite{BaumgartnerIF},
and others.  This theory is quite well understood and has been used in many
applications, most important of which is the consistency of the Proper Forcing Axiom (PFA),
obtained by Baumgartner \cite{BaumgartnerPFA} and Shelah \cite{Sh_P}.
In order to iterate semiproper forcing Shelah had to devise a new type of support,
called the {\em revised countable support}, and show that the iteration of semiproper forcing
notions using this support is also semiproper. The most spectacular application
of this theory is the relative consistency of Martin's Maximum (MM), which was
shown by Foreman, Magidor and Shelah \cite{FMS}.
However, the definition of revised countable support is quite intricate and,
despite several attempts such as \cite{Miyamoto} and \cite{Fuchs} to simplify it,
it still remains somewhat of a mystery.

The starting point for our work is a new iteration technique recently introduced by Neeman \cite{Neeman}
who used it to give an alternative proof of the consistency of PFA. Neeman's method
uses finite supports for the working part together with side conditions which are finite $\in$-chains of
elementary submodels of two types, countable and transitive. The side conditions also
have to be closed under intersection. Neeman's primary motivation was to generalize this type
of iteration in order to  obtain generalizations of PFA to higher cardinals.
However, even the basic properties of Neeman's technique can be used to obtain
new results or give simpler and unified proofs of known theorems, see e.g. \cite{VeVe}.

Now, it is quite natural to ask if Neeman's method can be adapted to iterate more general
classes of posets, such as semiproper forcing.
However, when attempting to do that one encounters serious difficulties.
The most important one comes from the fact that for non proper forcing the operations
of taking a generic extension and intersection of two models do not commute.
When attempting to resolve this difficulty we were lead to the notion of {\em virtual models}
which turns out to be more suitable in this context. The basic idea of virtual models is that they encode
the essential information relative to an initial segment of the iteration that is contained in
some actual models which do not yet appear as side conditions at that stage of the iteration.
The upshot of these considerations is that we now use only countable models and the operation
of taking the intersection is replaced by a better behaved operation of projection.
The added benefit is that we now get a true iteration, namely each stage of
the iteration is a complete suborder of all the later stages, which  was not the case
of Neeman's iteration.

The paper is organized as follows. In \S 1 we introduce the notion of virtual models
and establish some of their basic properties. In \S 2 we recall some  facts about
generic extensions and elementary submodels. Mostly we make definitions that will be relevant
for the iteration of semiproper forcing.
Finally, in \S 3 we present our iteration technique and show that if each iterand is semiproper
then so is the resulting forcing notion. If one wishes, one can then use this method to reprove
the consistency of MM or some other applications of revised countable support iteration.
The hope is that this technique is quite flexible and will allow us to obtain results that
cannot be proved by standard methods, in particular it may be possible to generalize it
for higher cardinal versions of semiproper forcing.
Finally, let us note that a related but somewhat different approach to iterating semiproper
forcing was proposed by Gitik and Magidor \cite{GM}.
%The reader is encouraged to compare the two methods.
In order to read this paper, only basic knowledge of iterated forcing is needed, such as
the one presented in \cite{JechST} or \cite{Kunen}. For all undefined terms we refer the reader
to one of these two monographs.

\section{Virtual models}

In this section we present the collection of models we shall use as side conditions.
We shall consider the language $\mathcal L$ obtained by adding a unary function
symbol $U$ to the standard language $\mathcal L_\epsilon$ of set theory.
Let us say that a structure $\mathcal A$ of the form $(A,\in,U^{\mathcal A})$ is \emph{admissible}
if $A$ is a transitive set, $U^{\mathcal A}$ is a function from $\ORD^A$ to $A$, and $\mathcal A$
satisfies $\ZFC$ in the expanded language $\mathcal L$.
We shall often abuse notation and refer to the structure $(A,\in,U^{\mathcal A})$ simply by $A$.
Suppose $\mathcal A$ is an admissible structure. If $\alpha$ is an ordinal in $A$,
we let $A_\alpha$ denote $A\cap V_\alpha$. Finally, we let
$$
E_{\mathcal A}=\{ \alpha \in A : (A_\alpha,\in,U^{\mathcal A}\! \restriction \! \alpha) \prec (A,\in,U^{\mathcal A})\}.
$$
Note that $E_{\mathcal A}$ is a closed, possibly empty, subset of $\ORD^A$. For $\alpha \in A$
we let ${\rm next}_{\mathcal A}(\alpha)$ be the least ordinal in $E_{\mathcal A}$ above $\alpha$,
if such an ordinal exists. Otherwise, we leave ${\rm next}_{\mathcal A}(\alpha)$ undefined.
We start with a simple technical lemma.

\begin{lemma}\label{E_A-cofinal} {Suppose $M$ is an elementary submodel of
an admissible structure $\mathcal A$.
Then $\sup ( E_{\mathcal A}\cap M) = \sup (E_{\mathcal A}\cap \sup (M\cap \ORD^A))$.}
\end{lemma}

\begin{proof} Suppose $\beta \in E_{\mathcal A}$ and $(M\cap \ORD^A)\setminus \beta$ is
nonempty. Let $\gamma$ be the least ordinal in $M\setminus \beta$.
 We show that $A_\gamma$ is an elementary submodel
of $\mathcal A$. Suppose otherwise, then by the Tarski-Vaught criterion,
 there is a tuple $\bar{x}\in A_\gamma$
and a formula $\varphi (y,\bar{x})$ such that $\mathcal A \models \exists y \varphi(y,\bar{x})$,
but there is no $y\in A_\gamma$ such that $\mathcal A \models \exists y \varphi(y,\bar{x})$.
Since $\gamma \in M$ and $M$ is an elementary submodel of $\mathcal A$,
there is such a tuple $\bar{x}\in A_\gamma\cap M$.
Now, $\gamma$ is the least ordinal in $M$ above $\beta$,
therefore $\bar{x}\in M\cap A_\beta$.
Since $A_\beta$ is an elementary submodel of $\mathcal A$,
there is $y'\in A_\beta$ witnessing that $A_\beta \models \varphi(y',\bar{x})$
and so $A \models \varphi(y',\bar{x})$.
Since $\beta \leq \gamma$, it follows that $y'\in A_\gamma$, a contradiction.
\end{proof}

\begin{def}\label{M(X)-definition} Suppose $M$ is an elementary submodel of an admissible
structure $\mathcal A$ and $X$ is a subset of $A$. Let
$$
{\rm Hull}(M,X) =\{ f(\bar{x}) : f\in M, \bar{x}\in X^{<\omega}, f \mbox{ is a function, and } \bar{x}\in \dom (f)\}.
$$
\end{def}

\begin{lemma}\label{Hull(M,X)-elementary} Suppose $\mathcal A$ is an admissible structure,
$M$ is an elementary submodel of $\mathcal A$ and $X$ is a subset of $A$.
Let $\delta$ be $\sup (M\cap \ORD)$ and suppose $X\cap A_\delta$ is nonempty.
Then ${\rm Hull}(M,X)$ is the least elementary submodel of $\mathcal A$ containing $M$ and $X\cap A_\delta$.
\end{lemma}

\begin{proof} For each $\gamma \in A$, let ${\rm id}_\gamma$ be the identity function
on $A_\gamma$. Clearly, if $\gamma \in M$ then ${\rm id}_\gamma \in M$. Therefore,
$X\cap A_\delta$ is a subset of ${\rm Hull}(M,X)$. Let $\gamma \in M$ be such that $X\cap A_\gamma$ is nonempty.
For each $z\in M$, the constant function $c_z$ defined on $A_\gamma$ is in $M$,
therefore $M$ is a subset of ${\rm Hull}(M,X)$. The minimality of ${\rm Hull}(M,X)$ is clear from the definition.
It remains to show that ${\rm Hull}(M,X)$ is an elementary submodel of $A$.
We check the Tarski-Vaught criterion for ${\rm Hull}(M,X)$ and $A$. Let $\varphi$ be a formula and
$a_1,\ldots,a_n \in {\rm Hull}(M,X)$ such that $A\models \exists u \varphi(u,a_1,\ldots,a_n)$.
Then we can find functions $f_1,\ldots, f_n\in M$ and tuples $\bar{x}_1,\ldots,\bar{x}_n \in X^{<\omega}$
such that $a_i=f_i(\bar{x}_i)$, for all $i$. If $D_i$ is the domain of $f_i$, this implies that $\bar{x}_i\in D_i$.
By regularity and the axiom of choice in $A$ we can find a function $g$ defined on $D_1 \times \ldots \times D_n$
such that for every $\bar{y}_1\in D_1,\ldots,\bar{y}_n \in D_n$, if there is $u$ such that
$A\models \varphi (u, f_1(\bar{y}_1),\ldots,f_n(\bar{y}_n))$ then
$g(\bar{y}_1,\ldots,\bar{y}_n)$ is such a $u$. Moreover, by elementarity of $M$,
we may assume that $g\in M$. Let $a =g(\bar{x}_1,\ldots,\bar{x}_n)$.
It follows that $a\in {\rm Hull}(M,X)$ and $A \models \varphi (a,a_1,\ldots,a_n)$.
Therefore, ${\rm Hull}(M,X)$ is an elementary submodel of $A$.
\end{proof}

Now, let us fix an inaccessible cardinal $\kappa$ and a function $U^{\mathcal V_\kappa}:\kappa \rightarrow V_\kappa$.
Let $\mathcal V_\kappa= (V_\kappa,\in,U^{\mathcal V_\kappa})$. Clearly, $\mathcal V_\kappa$ is admissible.
We shall abuse notation and write simply $V_\kappa$ instead of $\mathcal V_\kappa$ and
$U$ instead of $U^{\mathcal V_\kappa}$.  We shall also write
 $E$ instead of $E_{\mathcal V_\kappa}$ and ${\rm next}(\alpha)$ instead of ${\rm next}_{\mathcal V_\kappa}(\alpha)$.

\begin{defi}\label{A-def} Suppose $\alpha \in E$. We let ${\mathscr A}_\alpha$ denote the set of all
admissible structures $\mathcal A$ that are elementary end extensions of
$(V_\alpha,\in,U\! \restriction \! \alpha)$ and have the same cardinality as $V_\alpha$.
\end{defi}

Note that if $A \in \mathscr A_\alpha$ and $\alpha \in A$ then $E_{\mathcal A} \cap (\alpha +1) =E\cap (\alpha+1)$.
Of course, $E_{\mathcal A}$ may have some elements above $\alpha$ which are not in $E$.
Another important point is that $\mathscr A_\alpha$ is uniformly definable in $\mathcal V_\kappa$
with parameter $\alpha$.  Our plan is to define an iteration $(\mathbb{P}_\alpha : \alpha \in E^*)$,
where $E^*= E \cup \{ \alpha+1: \alpha \in E\}$. Once we have defined $\mathbb P_\alpha$,
we will define $\mathbb P_{\alpha +1}$ to be essentially $\mathbb P_\alpha \ast \dot{\mathbb Q}_\alpha$,
where $\dot{\mathbb Q}_\alpha$ is the $\mathbb P_\alpha$-name for a semiproper forcing given
by $U(\alpha)$.
We will have that the initial segment of the iteration,
$(\mathbb{P}_\gamma : \gamma \in E^*\cap (\alpha +1))$, is uniformly definable
in  $V_\kappa$ with parameter $\alpha$ and $U\rest \alpha$, for every $\alpha \in E$.
Now, if $\alpha \in E$ and $\mathcal A \in {\mathscr A}_\alpha$, the same definition can be applied in
$\mathcal A$ and we get an iteration $(\mathbb{P}^{\mathcal A}_\gamma: \gamma \in E_{\mathcal A})$.
By elementarity, we will then have that $\mathbb{P}_\gamma^{\mathcal A}=\mathbb{P}_\gamma$,
for every $\gamma \in E\cap \alpha$.

We are now ready to define the collection of models that we plan use as side conditions
in our iteration.

\begin{defi}\label{C-def} Suppose $\alpha \in E$. We let $\mathscr C_\alpha$ denote
the collection of all countable submodels $M$ of $V_\kappa$ such that, if we let $A={\rm Hull}(M,V_\alpha)$,
then $A$ is transitive and there exists a function $U^{\mathcal A}$  such that
the structure ${\mathcal A}=(A,\in,U^{\mathcal A})$ belongs to $\mathscr A_\alpha$.
We refer to the members of $\mathscr C_\alpha$ as the $\alpha$-{\em models}.
We write  $\mathscr C_{<\alpha}$ for $\bigcup \{ \mathscr C_\gamma : \gamma \in E\cap \alpha\}$ and
$\mathscr C_{\leq \alpha}$ for ${\mathscr C}_{< \alpha} \cup {\mathscr C}_\alpha$.
We write $\mathscr C_{\geq \alpha}$ for $\bigcup \{ \mathscr C_\gamma : \gamma \in E\setminus \alpha\}$.
Finally, we let $\mathscr C$ denote $\mathscr C_{< \kappa}$.
\end{defi}

\begin{remark} Note that if $M \in \mathscr C_\alpha$ then $\sup(M \cap \ORD)\geq \alpha$.
In general, $M$ is not elementary in $V_\kappa$, in fact, this only happens if $M \subseteq V_\alpha$.
Finally, note that we are not requiring that the function $U^{\mathcal A}$  be unique.
However, we will have that $U^{\mathcal A}\! \restriction \! \alpha = U \! \restriction \! \alpha$, and
that is all we care about.
\end{remark}

We plan to use members of $\mathscr C_{\leq \alpha}$ as side conditions in the forcing $\mathbb P_\alpha$.
This will guarantee that $\mathbb P_\alpha$ is small, i.e. has size less than ${\rm next}(\alpha)$.
We will use $\alpha$-models to control the working parts of conditions below $\alpha$.
Now, we need to ensure that $\mathbb P_\alpha$ is a complete suborder of $\mathbb P_\beta$,
for $\alpha < \beta$. So, if $M$ is a $\beta$-model appearing in some condition from $\mathbb P_\beta$,
we need to find an $\alpha$-model $N$ which has the same impact as $M$ on the iteration
up to $\alpha$. This motivates the following definition.

\begin{defi}\label{alpha-equiv-def} Suppose $M,N\in \mathscr C$ and $\alpha \in E$.
We say that $M$ and $N$ are $\alpha$-{\em isomorphic} and write $M\cong_{\alpha}N$ if
there is an isomorphism $\sigma$ between ${\rm Hull}(M,V_\alpha)$ and ${\rm Hull}(N,V_\alpha)$ such that
$\sigma [M]=N$. Of course, if such a $\sigma$ exists, it is unique.
\end{defi}

Clearly, $\cong_{\alpha}$ is an equivalence relation, for every $\alpha$. Note that if $M \in \mathscr C_\gamma$,
for some $\gamma < \alpha$, then the only model $\alpha$-isomorphic to $M$ is $M$ itself.
Suppose  $\alpha,\beta \in E$ and $\alpha \leq \beta$.
It is easy to see that, if $M,N\in \mathscr C$ are $\beta$-isomorphic,
then they are $\alpha$-isomorphic.  We will now see that, if $\alpha <\beta$, then
for every $\beta$-model $M$ there is a canonical representative of the $\cong_\alpha$-equivalence
class of $M$ which is an $\alpha$-model.

\begin{defi}\label{projection-models-def} Suppose $\alpha$ and $\beta$ are members of  $E$ with $\alpha \leq \beta$,
and $M$ be a $\beta$-model. Let $\overline{{\rm Hull}(M,V_\alpha)}$ be the transitive collapse of
${\rm Hull}(M,V_\alpha)$ and let $\pi$ be the collapsing map. We define  $M \! \restriction \! \alpha$ to be $\pi[M]$,
i.e. the image of $M$ under the collapsing map of ${\rm Hull}(M,V_\alpha)$.
\end{defi}

Note that $\overline{{\rm Hull}(M,V_\alpha)}$ belongs to $\mathscr A_\alpha$, so $M\! \restriction \! \alpha$ is
an $\alpha$-model which is $\alpha$-isomorphic to $M$. Note also that if $\beta=\alpha$ then
$M\! \restriction \! \alpha =M$, since ${\rm Hull}(M,V_\alpha)$ is already transitive.
The following is straightforward.

\begin{prop}\label{transitivity-projections} Suppose $\alpha,\beta,\gamma \in E$ with
$\alpha \leq \beta \leq \gamma$. Let $M\in \mathscr C_{\gamma}$. Then
$(M \! \restriction \! \beta)\! \restriction \! \alpha = M\! \restriction \! \alpha$.
\qed
\end{prop}

We also need to define a version of the membership relation, for every $\alpha$ in $E$.

\begin{defi}\label{membership-alpha-def} Suppose $M,N \in \mathscr C$. We let $M\in_\alpha N$
if there is $M'\in N$ such that $M'\cong_\alpha M$.
\end{defi}

Note that if $M\subseteq V_\alpha$ this simply means that $M\in N$. However, in general,
we may have $M \in_\alpha N$ even if the rank of $M$ is higher than the rank of $N$.
Note that $\in_\alpha$ is transitive on members of $\mathscr C$. We will need
the following simple fact.

\begin{prop}\label{projection-membership} Let $\alpha,\beta \in E$ with $\alpha \leq \beta$.
Suppose $M$ and $N$ are $\beta$-models and $M\in_\beta N$.
Then $M\! \restriction \! \alpha \in_\alpha N\! \restriction \! \alpha$.
\qed
\end{prop}

\section{The scaffolding}

In this section we recall some elementary facts about iterated forcing and make some definitions
that will be relevant for our iteration.
Let us fix a transitive model $(A,\in, \ldots)$ of $\ZFC$, possibly with some additional structure,
and a forcing notion ${\mathbb Q}\in A$.
Then $A^{\mathbb Q}$ denotes the class of all $\mathbb Q$-names, as defined in $A$.
If $G$ is an $A$-generic filter over ${\mathbb Q}$, we define the interpretation
of ${\mathbb Q}$-names by $G$ by $\in^*$-recursion as follows:
$$
\tau_{G}= \{ \sigma_{G} : \mbox{there is } p \in G \mbox{ such that } (p,\sigma)\in \tau\}.
$$
The generic extension $A[G]$ is equal to  $\{ \tau_{G}: \tau \in A^{\mathbb Q}\}$.
Fix now an elementary submodel $M$ of $A$ with ${\mathbb Q}\in M$. We do not assume
that $M$ belongs to $A$ or even $A[G]$.
We let
$$
M[G] = \{ \tau_{G} : \tau \mbox{ is a ${\mathbb Q}$-name and $\tau \in M$}\}.
$$
\noindent
Then $M[G]$ is an elementary submodel of $A[G]$. The following
definition is non standard.

\begin{defi}\label{M(G)-def}
Let $M(G)$ denote the trace of $M[G]$ on $A$, i.e.  $M[G] \cap A$.
%We let $M(\dot{G})$ denote the canonical $\mathbb Q$-name for $M(G)$.
\end{defi}

By a result of Laver \cite{LaverSF} and Woodin \cite{WoodinGM} the
ground model $A$ is definable in $A[G]$ from the parameter $\mathcal P^A(\lambda)$,
the power set of $\lambda$ as computed in $A$,
where $\lambda$ is the $A$ cardinality of $\mathbb Q$.
It follows that $M(G)$ is an elementary submodel of $A$,
although, of course, it may not belong to $A$.
Therefore it makes sense to define $M(G)[G]$.
%We denote by $M[\Gdot]$ the canonical name
%for $M[G]$ and, similarly, by $M(\Gdot)$ the canonical name
%for $M(G)$.
We will need the following easy facts, using the above notation.

\begin{lemma}\label{trace-model-generic} $M(G)[G]=M[G]$.
\end{lemma}

\begin{proof} Since $M\subseteq M(G)$, it suffices to show
$M(G)[G]\subseteq M[G]$.
Let $\tau\in M$ be a ${\mathbb Q}$-name for an element of
$A^{{\mathbb Q}}$.
Working in $M$ we can find a maximal antichain $\mathcal X$ in ${\mathbb Q}$ and, for
every $q\in \mathcal X$, a ${\mathbb Q}$-name $\sigma_q$ such that $q \forces \tau=\check{\sigma}_q$.
By the Maximality Lemma we can find in $M$ a single ${\mathbb Q}$-name $\sigma$ such that
$q \forces \sigma_q=\sigma$, for all $q\in A$.
It follows that $(\tau_{G})_{G}=\sigma_{G}$.
\end{proof}

\begin{coro}\label{two-models-generic} Suppose $N$ is another elementary
submodel of $A$ with $M \subseteq N \subseteq M(G)$.
Then $N[G]=M[G]$.
\qed
\end{coro}

%In our situation we will have that ${\mathcal P}(\mathbb Q)\subseteq A$,
%so $G$ is $A$-generic iff it is $V$-generic. Hence, if $M\in V$ we can define
%$M[G]$ and $M(G)$ in $V[G]$. We then let $M[\Gdot]$ and $M(\Gdot)$ denote the canonical
%$\mathbb Q$-names for these objects.

We now generalize Definition \ref{M(G)-def} to a situation where we have
not one forcing, but an iteration of forcing notions.
Suppose $X\subseteq \ORD \cap A$ and $(\mathbb P_\xi : \xi \in X)$ is an increasing
chain of posets in $A$ such that $\mathbb P_\xi$ is a complete suborder of $\mathbb P_\eta$,
for all $\xi,\eta\in X$ with $\xi < \eta$. Suppose moreover
that $(\mathbb P_\xi : \xi \in X \cap (\alpha +1))$ is uniformly definable in $A$
with parameter $\alpha$.
Suppose $\delta \in X$ and $G_\delta$ is an $A$-generic filter over $\mathbb P_\delta$
and  let $G_\alpha$ be $G_\delta \cap \mathbb P_\alpha$, for $\alpha \in X\cap (\delta+1)$.
%Given an elementary submodel $M$ of $A$ we would like to define the model $M(G_\delta)$.
It will be convenient to write $G_{<\alpha}$ for the union of the $G_\xi$, for $\xi \in X\cap \alpha$.
So, if $X\cap \alpha$ has a largest element, say $\beta$, then $G_{<\alpha}$
is just $G_\beta$. Fix, as before, an elementary submodel $M$ of $A$.
We define, by induction on $\alpha \in X\cap (\delta +1)$,
what it means for $\alpha$ to be attainable from $M$ by $G_{<\alpha}$
and we construct a model $M(G_\alpha)$.
The sequence $(M(G_\alpha):\alpha \in X\cap (\delta+1))$ will form an increasing
chain of elementary submodels of $A$ and each of these models will contain $M$ as a submodel.
We let $M(G_{<\alpha})$ denote the union of the $M(G_\xi)$, for $\xi \in X\cap \alpha$.
If $\alpha$ is the least point of $X$ then $M(G_{<\alpha})$ is just $M$.
By induction and the theorem on elementary chain of models,
 $M(G_{<\alpha})$ will also be an elementary submodel of $A$.

\begin{defi}\label{M(G-delta)-def} Suppose $\alpha \in X\cap (\delta +1)$.
We say that $\alpha$ is {\em attainable} from $M$ by $G_{<\alpha}$ if
$\alpha \in M(G_{<\alpha})$. If this is the case we let $M[G_\alpha]=M(G_{<\alpha})[G_\alpha]$
and we let $M(G_\alpha)=M[G_\alpha]\cap A$. Otherwise we let  $M(G_\alpha)$ be equal to $M(G_{<\alpha})$.
\end{defi}

The idea is that we start with $M$. If $\alpha$ is the least element of $X$ we ask
if $\alpha \in M$. If so, then $\mathbb P_\alpha \in M$ as well and we can define
the model $M[G_\alpha]$. We then let $M(G_\alpha)$ be the trace of $M[G_\alpha]$ on
the ground model, i.e. $A$. In general, we ask if $\alpha$ appears in the union
of the previous models, i.e. $M(G_{<\alpha})$. If so, then we know that $\mathbb P_\alpha$
belongs to $M(G_{<\alpha})$, as well. Therefore, we can define the model $M(G_{<\alpha})[G_\alpha]$.
We then let $M(G_\alpha)$ be the trace of this model on $A$. Otherwise we let $M(G_\alpha)$ be $M(G_{<\alpha})$.
Clearly, in this way, we obtain an increasing chain of elementary submodels of $A$.
Note also that the model $M(G_\alpha)$ only depends on $M$ and $G_\alpha$ and
does not depend on the future generic filters $G_\beta$, for $\beta >\alpha$.

We now define what it means for a condition to make an ordinal attainable from $M$.
Before that we need to make some definitions. First of all, our transitive model $A$ will belong to $\mathscr A_\delta$,
for some $\delta  \in E$. We will have a sequence of posets $(\mathbb P_\alpha: \alpha \in E^*\cap \delta)$
such that $\mathbb P_\xi$ is a complete suborder of $\mathbb P_\eta$, for all $\xi <\eta$.
Moreover, for every $\alpha \in E^*\cap \delta$ the chain $(\mathbb P_\xi: \xi \in E^*\cap (\alpha +1))$
will be uniformly definable in $(A,\in,U^{\mathcal A})$ with parameter $\alpha$.
$M$ will be an elementary submodel of $A$ in $V$. Notice that if $\alpha \in E^*\cap \delta$
then a filter $G_\alpha$ is $V$-generic over $\mathbb P_\alpha$ iff it is $A$-generic over $\mathbb P_\alpha$.
Therefore, we can perform the construction of $M(G_\alpha)$ and $M[G_\alpha]$ in
$V[G_\alpha]$. We let $M(\Gdot_\alpha)$ denote
the canonical $\mathbb P_\alpha$-name for $M(G_\alpha)$ and we let $M[\Gdot_\alpha]$
denote the canonical $\mathbb P_\alpha$-name for  $M[G_\alpha]$, provided this model
makes sense, i.e. if $\alpha$ is attainable from $M$ by $G_{<\alpha}$.
Finally, for every condition $p$ we let $\xi(p)$ be the least $\xi \in E^*\cap \delta$ such that
$p\in \mathbb P_\xi$.

\begin{defi}\label{attainable-condition} Suppose $p$ is a condition and $\alpha \leq \delta$.
We say that $p$ {\em makes $\alpha$  attainable from} $M$  if there is $\xi \leq \min (\xi(p),\alpha)$
and $q\in \mathbb P_\xi$ such that $p\leq q$ and $q\forces_{\mathbb P_\xi} \alpha \in M(\Gdot_\xi)$.
\end{defi}

In our situation there will be canonical projections $p\mapsto p\rest \xi$, for $\xi \in E^*\cap \delta$.
We will have that $p\leq p\rest \xi$, for every $p$ belonging to $\mathbb P_\alpha$, for some $\alpha \in E^*\cap \delta$.
Moreover, if $q\in \mathbb P_\xi$ then $p\leq q$ iff $p\rest \xi \leq q$, and
$q$ and $p$ are compatible iff $q$ and $p\rest \xi$ are compatible.
The point then is that, if $p$ makes some $\alpha \in E^*\cap \delta$ attainable from $M$ and
$G_\alpha$ is $V$-generic over $\mathbb P_\alpha$ such that $p\rest \alpha \in G_\alpha$,
then we can define the model $M[G_\alpha]$.

We now review some definitions involving semiproper forcing and place them in our context.
Suppose $\mathbb Q$ is a forcing notion and let $\theta$ be a sufficiently large regular cardinal.
Let $M$ be a countable elementary submodel of $H(\theta)$ such that $\mathbb Q\in M$.
We say that a condition $q\in \mathbb Q$ is $(M,\mathbb Q)$-{\em semigeneric} if:
\[
q \forces_{\mathbb Q} M[\Gdot]\cap \omega_1=M\cap \omega_1,
\]
where $\Gdot$ is the canonical name for the $V$-generic filter over $\mathbb Q$.
Of course, here we can also write $M(\Gdot)$ instead of $M[\Gdot]$.
We say that $\mathbb Q$ is {\em semiproper} if for every such $M$ and every $p\in M\cap \mathbb Q$
there is $q\leq p$ which is $(M,\mathbb Q)$-semigeneric. The important point is that
semiproper forcing notions preserve $\omega_1$, indeed, they preserve stationary subsets of $\omega_1$.
Suppose now $\mathbb Q\in V_\alpha$, for some $\alpha \in E$, and let $M$ be an $\alpha$-model
such that $\mathbb Q \in M$. Note that a condition $q$ is $(M,\mathbb Q)$-semigeneric iff
it is $(M\cap V_\alpha,\mathbb Q)$-semigeneric, so we can reformulate semiproperness in terms
of the existence of semigeneric conditions for all $\alpha$-models containing $\mathbb Q$.
We will often use without mentioning the following fact.

\begin{prop}\label{commute} Suppose $\alpha,\beta \in E$ with $\alpha <\beta$.
Let $\mathbb Q\in V_\alpha$ be a forcing notion and $M$ a $\beta$-model such that $\mathbb Q\in M$.
Let $G$ be a $V$-generic filter over $\mathbb Q$. Then $M(G) \cong_\alpha (M\rest \alpha)(G)$.
\qed
\end{prop}

\section{The iteration}

We now describe our iteration. Let us fix a function $U:\kappa \rightarrow V_\kappa$.
We refer to $U$ as the bookkeeping device. For instance, if $\kappa$ is supercompact
$U$ could be the Laver function for $\kappa$. Let $E^*= E \cup \{\alpha +1: \alpha \in E\}$.
Our plan is to define an increasing sequence of
semiproper posets $(\mathbb P_\alpha : \alpha \in E^*)$ such that

\begin{itemize}
\item[$(i)$] $\mathbb P_\alpha$ is of size $< {\rm next}(\alpha)$ and is uniformly definable in $V_\kappa$,
 with parameters $\alpha$ and $U\! \restriction \! \alpha$, for every $\alpha \in E$,
\item[$(ii)$] if $\alpha \in E$ and $U(\alpha)$ is a $\mathbb P_\alpha$-name $\dot{\mathbb Q}_\alpha$ for a semiproper forcing
notion then $\mathbb P_{\alpha+1}$ is isomorphic to $\mathbb P_{\alpha}\ast \dot{\mathbb Q}_\alpha$,
\item[$(iii)$] if $\alpha, \beta \in E^*$ and $\alpha <\beta$ then $\mathbb P_\alpha$ is a complete suborder of $\mathbb P_\beta$,
 \item[$(iv$)] if $\alpha$ is a limit of $E$ and is either inaccessible or of cofinality $\omega_1$ then $\mathbb P_\alpha$ is
 equivalent to the direct limit of $(\mathbb P_\gamma : \gamma\in E^* \cap \alpha)$.
 \end{itemize}

We then let $\mathbb P_\kappa$ be the direct limit of  $(\mathbb P_\alpha : \alpha \in E^*)$.
Condition $(iv)$ is not really needed for the construction itself.
It comes in later in order to prove the $\kappa$-c.c. of $\mathbb P_\kappa$.
It is also important in $\diamondsuit$-like arguments. For instance, if $U$ is a Laver function
and we wish our iteration to capture all semiproper forcing notions in $V^{\mathbb P_\kappa}$.
We will say a few more words about this point later.

Before we start our construction let us make some preliminary remarks and definitions.
Suppose $\alpha \in E^*$ and we have defined the $\mathbb P_\xi$, for $\xi \in E^*\cap (\alpha+1)$,
satisfying conditions $(i)$-$(iii)$. Suppose $\delta \in E\setminus \alpha$ and $M\in \mathscr C_{\geq \delta}$.
Recall that if $r$ is a condition then $\xi(r)$ denotes the least $\xi\in E^*$ such that $r\in \mathbb P_\xi$.

\begin{defi}\label{active} Suppose $r\in \mathbb P_\alpha$, $\delta \in E\setminus \alpha$,
and $M\in \mathscr C_{\geq \delta}$.  We say that $r$ {\em makes $M$ active at $\delta$}
 if there is some $\xi \leq \xi(r)$ which is made attainable from $M$ by $r$ such
 that
 \[
r \forces_{\mathbb P_{\xi(r)}} M[\dot{G}_\xi] \cap E \cap \delta \mbox{ is cofinal in } E\cap \delta.
\]
\end{defi}

Notice that if $\delta$ is a successor point of $E$, say $\delta ={\rm next}(\gamma)$, for $\gamma \in E$,
this simply means that $r$ makes $\gamma$ attainable from $M$.
If $\delta$ is a limit point of $E$, by Lemma \ref{E_A-cofinal}, this means
that $r$ forces that $M[\Gdot_\xi]\cap \delta$ is cofinal in $\delta$.

Now, suppose $\M$ is a finite subset of $\mathscr C$.
We let
\[
\M^{r,\delta} = \{ M\restriction \delta : M\in \M \cap \mathscr C_{\geq \delta}\mbox{ and $r$ makes $M$ active at $\delta$} \}.
\]
Notice that if $s\leq r$ then $\M^{r,\delta}$ is a subset of $\M^{s,\delta}$.
Let $\dot{\M}^{\alpha,\delta}$ be the $\mathbb P_\alpha$-name for the union of the $\M^{r,\delta}$,
for $r$ in the generic filter $G_\alpha$ over $\mathbb P_\alpha$.
Since $\M$ is finite there is always going to be a condition $r\in G_\alpha$ which {\em decides}
$\dot{\M}^{\alpha,\delta}$, i.e. such that $\M^{s,\delta}=\M^{r,\delta}$, for all $s\leq r$ with $s\in \mathbb P_\alpha$.
If $\delta={\rm next}(\alpha)$ we will write $\dot{\M}^{\delta}$ instead of $\dot{\M}^{\alpha,\delta}$.
We will one more definition.

\begin{defi}\label{weak-in-chain}
Let  $\M$ is a finite collection of $\delta$-models and let $G_\alpha$ be $V$-generic over $\mathbb P_\alpha$.
We say that $\M$ is a {\em weak $\in_\delta$-chain at $\alpha$} if, for every $M,N\in \M$,
we have:
\begin{enumerate}
\item if $M \cap \omega_1 = N\cap \omega_1$ then $M= N$,
\item if $M\cap \omega_1 < N\cap \omega_1$ then $M\in_\delta N(G_\alpha)$.
\end{enumerate}
\end{defi}

This definition is of course made in the generic extension $V[G_\alpha]$ and depends
on our choice of the generic filter $G_\alpha$.
Notice that if a condition $p\in \mathbb P_\alpha$ forces $\M$ to be a weak $\in_\delta$-chain at $\alpha$,
then it forces it to be a weak $\in_\delta$-chain at $\alpha'$, for any $\alpha'\in E^*$
such that $\alpha \leq \alpha' \leq \delta$.

We now outline our construction.
We will define the $\mathbb P_\alpha$ by induction on $\alpha$. Once we have defined
$\mathbb P_\alpha$, for $\alpha \in E$, we  define a $\mathbb P_\alpha$-name $\dot{\mathbb Q}_\alpha$.
If $U(\alpha)$ is a $\mathbb P_\alpha$-name for a semiproper forcing notion,
we let $\dot{\mathbb Q}_\alpha$ be equal to $U(\alpha)$, otherwise we let $\dot{\mathbb Q}_\alpha$
be the canonical $\mathbb P_\alpha$-name for the trivial forcing notion $\mathbbm{1}$.
If $\alpha \in E$, let us say that $p$ is an $\alpha$-{\em pair} if it is of the form $(\M_p,w_p)$, where $\M_p$
 is a finite collection of models from $\mathscr C_{\leq \alpha}$ and $w_p$ is a finite function with $\dom(w_p)$
 a subset of $E \cap \alpha$. The forcing notion $\mathbb P_\alpha$ will consist of certain $\alpha$-pairs.
Once we have defined $\mathbb P_\alpha$ we define $\mathbb P_{\alpha+1}$ as the set of all pairs
$(\M_p,w_p)$, where $\M_p$ is a finite subset of $\mathscr C_{\leq \alpha}$ and $w_p$ is a finite function with
 $\dom(w_p)$ a subset of $E\cap (\alpha+1)$ such that $(\M_p,w_p\! \restriction \! \alpha)\in \mathbb P_\alpha$
 and, if $\alpha \in \dom(w_p)$, then $w_p(\alpha)$ is a canonical $\mathbb P_{\alpha}$-name
 for a condition in $\dot{\mathbb Q}_\alpha$. The order on $\mathbb P_{\alpha +1}$ is defined
 by letting $q\leq p$ if $(\M_q,w_q\! \restriction \! \alpha) \leq \M_p,w_p\! \restriction \! \alpha)$ and,
 if $\alpha \in \dom(w_p)$ then $\alpha \in \dom(w_q)$ and
 $(\M_q,w_q\! \restriction \! \alpha) \forces_{\mathbb P_\alpha} w_q(\alpha)\leq w_p(\alpha)$.
 Thus, $\mathbb P_{\alpha+1}$ will be canonically isomorphic to $\mathbb P_\alpha \ast \dot{\mathbb Q}_\alpha$.
 If $p$ is  an $\alpha$-pair  and $\gamma \in E\cap \alpha$ we will
let $p\! \restriction \!  \gamma$ denote the pair $(\M_p \rest \gamma, w_p\rest \gamma)$
where $\M_p$ is the collection $\{ M \rest \gamma : M \in \M_p\}$ and $w_p\rest \gamma$
is simply the restriction of $w_p$ to $\dom(w_p)\cap \gamma$. Similarly, we will let $p\! \restriction \! (\gamma +1)$
denote $(\M_p\rest \gamma,w_p\! \restriction \! \gamma +1)$.
It will be immediate from the construction that if $p \in \mathbb P_\alpha$ then
$p\! \restriction \! \gamma \in \mathbb P_\gamma$, for all $\gamma \in E^*\cap \alpha$.
In fact, the map $p \mapsto p \! \restriction \!  \gamma$
will be a canonical projection of $\mathbb P_\alpha$ to $\mathbb P_\gamma$.
In order for a pair $p$ to be in $\mathbb P_\alpha$ there will be two types of requirements.
First, if $\gamma \in \dom(w_p)$ we will require that $w_p(\gamma)$ be forced by $p\restriction \gamma$ to be
$(M[\Gdot_\gamma],\dot{\mathbb Q}_\gamma)$-semigeneric, for all $M\in \M_p$ for which
this makes sense. The second is that for each $\delta \in E\cap (\alpha +1)$, if
$\gamma\in E^*\cap \delta$ is sufficiently large, then $p\restriction \gamma$ forces
$\dot{\M}_p^{\gamma,\delta}$ to be a weak $\in_\delta$-chain at $\gamma$. Of course, if $\delta$
is a successor point of $E$, say $\delta ={\rm next}(\gamma)$, we may simply say  that
$p\restriction (\gamma +1)$ forces  $\dot{\M}_p^\delta$ to be a weak $\in_\delta$-chain at $\gamma +1$.
We will need the second condition in order to be able to extend $w_p$ to any $\gamma \in E\cap \alpha$
and also to prove the semiproperness of the iteration.

With these remarks in mind we are ready to define the posets $\mathbb P_\alpha$.
Of course, formally the definition is by induction on $\alpha$.

\begin{defi}\label{main-def} Suppose $\alpha \in E$.  We say that a pair $p$ of the form $(\M_p,w_p)$
belongs to $\mathbb P_\alpha$ if $\M_p$ is a finite subset of $\mathscr C_{\leq \alpha}$,
$w_p$ is a finite function with domain included in $E\cap \alpha$ and:
\begin{enumerate}
\item for every $\delta \in E\cap (\alpha +1)$ there exists $\bar{\delta} \in E\cap \delta$
such that, if $\gamma \in E^*$ and $\bar{\delta} <\gamma <\delta$, then
\[
p\restriction \gamma \forces_{\mathbb P_\gamma} \dot{\M}_p^{\gamma,\delta} \mbox{ is a weak }
\in_\delta \! \! \mbox{-chain at } \gamma.
\]
\item If $\gamma \in \dom(w_p)$ and $\delta ={\rm next}(\gamma)$ then
\[
p\! \restriction \! \gamma \forces_{\mathbb P_{\gamma}} w_p(\gamma) \mbox{ is }
(M[\Gdot_{\gamma}],\dot{\mathbb Q}_\gamma)\mbox{-semigeneric, for all } M\in \dot{\M}^\delta_p.
\]
\end{enumerate}
The order is defined as follows: $q \leq p$ if for every $\gamma \in E\cap (\alpha +1)$ and a $\gamma$-model
$M\in \M_p$ there is $N\in \M_q$ such that $M= N\! \restriction \! \gamma$, $\dom(w_p)\subseteq \dom(w_q)$,
and, for every $\gamma \in \dom(w_p)$,
\[
q\! \restriction \! \gamma \forces_{\mathbb P_\gamma} w_q(\gamma)\leq w_p(\gamma).
\]
\end{defi}

Once we have $\mathbb P_\alpha$, we define $\mathbb P_{\alpha +1}$, for $\alpha \in E$, as described above.
Clearly, the order relation $\leq$ is transitive on $\mathbb P_\alpha$, for $\alpha \in E^*$.
Suppose $\alpha,\beta \in E^*$ and $\alpha <\beta$.
It follows  from the definition  that $\mathbb P_\alpha$ is contained  in $\mathbb P_\beta$
and the order relation on $\mathbb P_\alpha$ is the restriction of the order relation on $\mathbb P_\beta$.
Suppose $p\in \mathbb P_\beta$, $q\in \mathbb P_\alpha$ and $q\leq p\! \restriction \! \alpha$.
Then we can define a condition $r$  as follows. Set $\M_r=\M_p\cup \M_q$
and let $w_r = w_q \cup w_p\! \restriction [\alpha,\beta)$.
It is straightforward to check that $r\in \mathbb P_\beta$ and is a greatest lower bound of $p$ and
$q$. We will denote  $r$ by $p\meet q$. It follows that the map $p\mapsto p\! \restriction \! \alpha$
is a projection from $\mathbb P_\beta$ to $\mathbb P_\alpha$.
It is easy to see that the set of conditions $p$ in $\mathbb P_\alpha$ such that $p\! \restriction \! \gamma$
fixes $\M_p$ at $\gamma$, for every $\gamma \in \dom(w_p)$, is dense in $\mathbb P_\alpha$.
It will be convenient to write $\mathbb P_{<\alpha}$ for the union of the $\mathbb P_\xi$, for $\xi \in E^*\cap \alpha$.
Also, if $p\in \mathbb P_{<\kappa}$ and $\xi \in E^*$ we write
$\mathbb P_\xi\rest p$ for the set of all $q\in \mathbb P_\xi$ extending $p\rest \xi$.
Finally, we write $\mathbb P_{<\alpha}\rest p$ for the union of the $\mathbb P_{\xi}\rest p$,
for $\xi \in E^*\cap \alpha$.

\begin{lemma}\label{adding-M} Suppose $\alpha \in E$ and $p$ is a condition in $\mathbb P_\alpha$.
Let $M$ be a $\beta$-model, for some $\beta >\alpha$, and suppose that  $p\in M$. Then there is
a condition $p^M\leq p$ such that $M\rest \alpha \in \M_{p^M}$.
\end{lemma}

\begin{proof} Let $\M_{p^M} = \M_p\cup \{ M\! \restriction \! \alpha \}$. Let $\gamma \in \dom(w_p)$.
By Proposition \ref{projection-membership}, we have that if $N\in \M_p$ then
$N\! \restriction \! \gamma \! \in_\gamma M\! \restriction \! \gamma$, for every $\gamma \in E\cap (\alpha +1)$.
We now define $w_{p^M}$. We will have $\dom(w_{p^M})=\dom(w_p)$.
Suppose $\gamma \in \dom(w_p)$ and note that  $\gamma$ as well as $w_p(\gamma)$ are in  $M$.
Since $U \! \restriction \! (\gamma +1) \in M$, it follows that  $\mathbb P_\gamma$ and $\dot{\mathbb Q}_\gamma$ also
belong to $M$.  Now $\dot{\mathbb Q}_\gamma$ is forced to be semiproper and $w_p(\gamma)\in M$,
so we can find a $\mathbb P_\gamma$-name $w_{p^M}(\gamma)$ for a condition in $\dot{\mathbb Q}_\gamma$
which is forced by $p\! \restriction \! \gamma$ to extend $w_p(\gamma)$ and be
$(M[\Gdot_\gamma],\dot{\mathbb Q}_\gamma)$-semigeneric. If we let $M'=M\! \restriction \! {\rm next}(\gamma)$,
then clearly $p\! \restriction \! \gamma$ also forces that $w_{p^M}(\gamma)$ is
$(M'[\Gdot_\gamma],\dot{\mathbb Q}_\gamma)$-semigeneric. Finally, we let $p^M=(\M_{p^M},w_{p^M})$.
It is straightforward to check that $w_{p^M}$ is as required.
\end{proof}

 \begin{lemma}\label{extending-w_p} Let $\gamma \in E$ and $\delta ={\rm next}(\gamma)$.
 Suppose $p\in \mathbb P_\delta$  and $\gamma \notin \dom(w_p)$. Suppose also $M\in \M_p$
 and $\dot{w}$ is a $\mathbb P_\gamma$-name for a condition in $\dot{\mathbb Q}_\gamma$.
Suppose that $p\rest \gamma \forces_{\mathbb P_\gamma} M\in \dot{\M}_p^\delta \mbox{ and } \dot{w}\in M[\Gdot_\gamma]$.
Then there is a $\mathbb P_\gamma$-name $\dot{w}^*$ for a condition in $\dot{\mathbb Q}_\gamma$
such that $p\rest \gamma \forces_{\mathbb P_\gamma} \dot{w}^*\leq \dot{w} \mbox{ and } \dot{w}^* \mbox{ is }
(N[\Gdot_\gamma],\dot{\mathbb Q}_\gamma)\mbox{-semigeneric, for all } N\in \dot{\M}_p^\delta$,
such that $M\cap \omega_1 \leq N\cap \omega_1$.
 \end{lemma}

\begin{proof}  Let $G_\gamma$ be any $V$-generic filter over $\mathbb P_\gamma$ containing
$p\! \restriction \! \gamma$. Let $\M$ be the interpretation of $\dot{\M}_p^\delta$ by $G_\gamma$
and let $\{ M_0,\ldots,M_{n-1}\}$ be the enumeration of $\M$ in the increasing order of the intersections
with $\omega_1$.  Then $M_i\in_{\delta} M_j[G_\gamma]$, for all $i<j$.
Let $\mathbb Q_\gamma$ be the interpretation of $\dot{\mathbb Q}_\gamma$ by $G_\gamma$.
 Then $\mathbb Q_\gamma$ is a semiproper forcing notion and is definable from
$U\restriction (\gamma +1)$. Therefore it belongs to $M_i[G_\gamma]$, for all $i$.
Suppose $M=M_k$ and let $w$ be the interpretation of $\dot{w}$ by $G_\gamma$.
Then $w\in M_k[G_\gamma]$. Let $\lambda$ be the cardinality of  ${\mathbb Q}_\gamma$ and assume,
without loss of generality, that the domain of $\mathbb Q_\gamma$ is $\lambda$.
Let $N_i=M_i[G_\gamma] \cap H({2^\lambda}^+)$, for all $i$.
Note that a condition is $(M_i[G_\gamma],\mathbb Q)$-semigeneric iff it is $(N_i[G_\gamma],\mathbb Q)$-semigeneric.
Let $\N= \{ N_i: k\leq i < n\}$.
Then $\N$ is an actual $\in$-chain of countable models containing $\mathbb Q_\gamma$.
We let $w_k=w$ and use the definition of semiproperness to build a decreasing chain of conditions
$w_k\geq w_{k+1}\geq \ldots \geq w_n$,  such that $w_i \in N_i$ and
$w_{i+1}$ is   $(N_i[G_\gamma],{\mathbb Q}_\gamma)$-semigeneric, for $k\leq i<n$.
Let  $w^*=w_n$ and pick a  canonical $\mathbb P_\gamma$-name $\dot{w}^*$ for $w^*$.
Since $G_\gamma$ was an arbitrary generic filter over $\mathbb P_\gamma$ containing $p\! \restriction \! \gamma$,
we have that  $\dot{w}^*$ is as required.
\end{proof}

We are now ready to state the main technical result.

\begin{teo}\label{main-thm} Suppose $\gamma \in E$ and  $p\in \mathbb P_\gamma$.
Let $M$ be an $\alpha$-model, for some $\alpha >\gamma$, and suppose that $M\! \restriction \! \gamma \in \M_p$.
Then $p \forces_{\mathbb P_\gamma}M(\Gdot_\gamma)\cap \omega_1= M\cap \omega_1$.
\end{teo}

\begin{proof} We show this by induction on $\gamma$. The basic idea is that the only ground model objects
that are added to $M$ by $G_\gamma$ are added by $G_{{\mathbb Q}_\xi}$, for some $\xi \in E\cap \gamma$,
where $G_{{\mathbb Q}_\xi}$ is the generic over $\mathbb Q_\xi$  induced by $G_\gamma$.
We proceed by cases according to whether $\gamma$ is a successor or a limit point of $E$.

\medskip

\noindent {\bf Case 1.} Let us first assume that $\gamma$ is a successor
point of $E$, say $\gamma ={\rm next}(\beta)$, for $\beta \in E$.
Let us pick an arbitrary $V$-generic filter $G_\beta$ over $\mathbb P_\beta$ containing $p\rest \beta$
and show that $M(G_\gamma)\cap \omega_1= M(G_\beta)\cap \omega_1$, for any
$V$-generic filter $G_\gamma$ over $\mathbb P_\gamma$ which extends $G_\beta$ and contains $p$.
We may assume that $\gamma$ and hence also $\beta$ belongs to $M(G_\beta)$,
otherwise we have that $M(G_\gamma)= M(G_\beta)$. Also, we may assume that $\beta\in \dom(w_p)$.
%Let $\mathbb Q_\beta = {\rm val}_{G_\beta}(\dot{\mathbb Q}_\beta)$.
Since $M\rest \gamma \in {\rm val}_{G_\beta}(\dot{\M}_p^\gamma)$ we have that ${\rm val}_{G_\beta}(w_p(\beta))$ is
$(M[G_\beta],\mathbb Q_\beta)$-semigeneric, where $\mathbb Q_\beta$ is ${\rm val}_{G_\beta}(\dot{\mathbb Q}_\beta)$.
Hence, if $G_{\mathbb Q_\beta}$ is $V[G_\beta]$-generic over $\mathbb Q_\beta$ containing
${\rm val}_{G_\beta}(w_p(\beta))$ and
$G_{\beta+1}$ is the associated $V$-generic over $\mathbb P_{\beta +1}$, then
$M(G_{\beta+1})\cap \omega_1= M(G_\beta)\cap \omega_1$ and so, by the inductive hypothesis,
$M(G_{\beta+1})\cap \omega_1= M\cap \omega_1$.
Let $\mathbb R_\gamma$ be the quotient forcing $\mathbb P_\gamma/G_\beta$.
We will be done once we establish the following.

\begin{cla}\label{successor-generic} $p\forces_{\mathbb R_\gamma} M(\Gdot_\gamma)= M(\Gdot_{\beta+1})$.
\end{cla}

\begin{proof}
Let $\tau \in M[G_\beta]$ be an $\mathbb R_\gamma$-name for an element of
the ground model and let $D$ be the set of conditions deciding the value of $\tau$.
Then $D$ is a dense open subset of $\mathbb R_\gamma$.
It suffices to show that for every $q\in \mathbb R_\gamma$ with $q\leq p$  there is a condition $r\in D$
and  $s\leq q,r$ such that $s\! \restriction \! (\beta+1) \forces_{\mathbb Q_\beta} r \in M(\Gdot_{\beta+1})$.
Then $M(\Gdot_{\beta+1})$ will be forced by $s$ to contain the value of $\tau$, as decided by $r$.
Consider now one such $q\leq p$ and assume, without loss of generality, that $q\in D$.
Let $w={\rm val}_{G_\beta}(w_q(\beta))$  and let $\M = {\rm val}_{G_\beta}(\dot{\M}_q^{\gamma})$.
Since $q$ is a condition, we know that $q\rest (\beta +1)$ forces $\M$ to be a weak $\in_\gamma$-chain at $\beta +1$.
Hence, for every $Q,R\in \M$, if $Q\cap \omega_1=R\cap \omega_1$ then $Q=R$ and
if $Q\cap \omega_1< R\cap \omega_1$ then $w\forces_{\mathbb Q_\beta}Q\in_{\gamma}R(\Gdot_{\beta +1})$.
Let $\{ M_0,\ldots, M_{n-1}\}$ be the enumeration of the members of $\M$ in the order given by their intersections
with $\omega_1$. Suppose $M\rest \gamma=M_k$ and let $\N= \{ M_0,\ldots,M_{k-1}\}$.
%Note that $w\forces_{\mathbb Q_\beta} M_i\in_\gamma M_j(\Gdot_{\beta+1})$, for all $i<j$.
We claim that $w\forces_{\mathbb Q_\beta} M_i\in M(\Gdot_{\beta+1})$, for all $i<k$,
and hence $w\forces_{\mathbb Q_\beta} \N \in M(\Gdot_{\beta+1})$.
To see that, pick an arbitrary $V[G_\beta]$-generic filter $G_{\mathbb Q_\beta}$ over $\mathbb Q_\beta$ containing
$w$ and let $G_{\beta+1}$ be the associated $V$-generic filter over $\mathbb P_{\beta +1}$.
For each $i<k$ we know that $M_i \in_\gamma \! \! (M\! \restriction \! \gamma)(G_{\beta+1})$, so there is
$M_i^\#\in M(G_{\beta+1})$ such that $M_i^\# \! \restriction \! \gamma = M_i$.
However, $\gamma \in M(G_{\beta+1})$, so we can compute $M_i^\# \! \restriction \! \gamma$
inside $M(G_{\beta+1})$. It follows that $M_i \in M(G_{\beta+1})$, for all $i<k$.
So, now we can find  a $\mathbb Q_{\beta}$-name $\dot{\N}\in M[G_\beta]$ for a finite collection of $\gamma$-models
such that $w \forces_{\mathbb Q_\beta} \dot{\N}= \check{\N}$.
We define  $D^*$ to be the set of all $u\in \mathbb Q_\beta$ such that
\medskip
\begin{itemize}
\item[$(*)$] there is $r\in D$ such that ${\rm val}_{G_\beta}(w_r(\beta))=u$ and
$u\forces_{\mathbb Q_\beta} \dot{\N} \subseteq \M_r$.
\end{itemize}
\medskip
\noindent  Clearly, $D^*$ is an open subset of $\mathbb Q_\beta$.
It may not be dense, so we let $(D^*)^{\perp}$ be the collection
of all $u\in \mathbb Q_\beta$ that are incompatible with all members of $D^*$ and we let
$D^{**}=D^*\cup (D^*)^{\perp}$. Then $D^{**}$ is a dense subset of ${\mathbb Q_\beta}$.
Working in $M[G_\beta]$, fix a maximal antichain $A$ contained in $D^{**}$ and, for each $u\in A\cap D^*$,
a condition $r_u\in D$  witnessing  ($*$) for $u$.  Let $\dot{r}=\{ (u,\check{r}_u) : u \in A\cap D^*\}$.
Then $\dot{r}$ is a $\mathbb Q_\beta$-name which is forced to be either the empty set or
 an element of $D$. Moreover,  $\dot{r}\in M[G_\beta]$.
Let $u$ be an element of $A$ which is compatible with $w$ and let $w^*$ be a common extension of
$w$ and $u$. Let us show that $u\in D^*$.  To see this, it suffices to show that $u\notin (D^*)^{\perp}$.
Otherwise $w^*$ would also be in $(D^*)^{\perp}$, but $w^*$ extends $w$ which is in $D^*$ as witnessed
by $q$ itself, a contradiction. Notice that $w^*\forces_{\mathbb Q_\beta} \dot{r}=\check{r}_u$.
To simplify notation, let $r=r_u$.  We claim that $q$ and $r$ are compatible in $\mathbb R_\gamma$.
To see this, first pick a $\mathbb P_\beta$-name $\dot{w}^*$ for $w$
and let $t$ be a common extension of $q$ and $r$ in $G_\beta$ which forces
all the relevant facts about $q$, $r$ and $\dot{w}^*$. In particular, $t$ decides $\dot{\M}_q^{\gamma}$ and
$\dot{\M}_{r}^{\gamma}$.
Now, let $\M_s=\M_t \cup \M_q\cup \M_r$, $w_s= w_t \cup \{ (\beta,\dot{w}^*)\}$  and let  $s=(\M_s,w_s)$.
We check that $s\in \mathbb P_\gamma$. Since $s\rest \beta=t$
which is in $G_\beta$ and extends $q\rest \beta$ and $r\rest \beta$ this implies that $s\in \mathbb R_\gamma$.
It suffices to verify that Definition \ref{main-def} holds for $s$ at $\gamma$.
Now, $t$ forces in ${\mathbb P_\beta}$  that $\dot{\M}_s^{\gamma}= \dot{\M}_q^{\gamma}\cup \dot{\M}_r^{\gamma}$
and it also forces that $\dot{w}^*$ is below $w_q(\beta)$ and below $w_r(\beta)$.
Hence it forces that $\dot{w}^*$ is $(N[\Gdot_\beta],\dot{\mathbb Q}_\beta)$-semigeneric,
for all $N\in \dot{\M}_s^{\gamma}$. This verifies condition (2) of Definition \ref{main-def}.
To see that condition (1) holds take $Q\in \M_q^{t,\gamma}$ and $R\in \M_r^{t,\gamma}$.
If $Q\cap \omega_1 < M\cap \omega_1$ then $Q$ belongs to $\N$ and
$s\rest (\beta +1) \forces_{\mathbb P_{\beta+1}} \N \subseteq \M_r$,
so condition (1) holds for $Q$ and $R$ since $r\in \mathbb P_\gamma$.
If $Q\cap \omega_1 = M \cap \omega_1$ then $Q\cong_\gamma M$ and we know
that $t\forces_{\mathbb P_\beta} \M_r \subseteq M(\Gdot_\beta)$.
Hence $t \forces_{\mathbb P_\beta} R \in_\gamma Q(\Gdot_\beta)$.
Finally, if $M\cap \omega_1 < Q\cap \omega_1$
then $s\rest (\beta+1)\forces_{\mathbb P_{\beta +1}} M\in_\gamma Q(\Gdot_{\beta +1})$
and hence $s\rest (\beta+1)\forces_{\mathbb P_{\beta +1}} R\in_\gamma Q(\Gdot_{\beta +1})$,
 by the transitivity of $\in_\gamma$.
Therefore, we have condition (1) in this case as well.
Finally, it is clear from the definition that $s$ extends $q$ and $r$.
This establishes Claim \ref{successor-generic} and completes the proof in {\bf Case 1}.

\end{proof}

\medskip

\noindent {\bf Case 2.} Assume now that  $\gamma$ is a limit point of $E$. Before we start we need to establish
a technical fact.
%let us prove a technical lemma.
%Recall that by the inductive hypothesis we know that $\mathbb P_\beta$ is semiproper, for all $\beta < \gamma$.

\begin{cla}\label{stability}
%Suppose $\delta\leq \gamma$ is a limit point of $E$, $q\in \mathbb P_\delta$ and $N\in \M_q$.
%Then there is $\bar{\delta} \leq \delta$ which is also a limit of $E$, $\beta \in E\cap \bar{\delta}$ and
%$t\leq q\rest \beta$ which makes $\beta$ attainable from $N$ and
%$(q\wedge t)\rest \eta \forces_{\mathbb P_\eta} \sup(N(\Gdot_\eta)\cap \delta)= \bar{\delta}$, for all $\eta \in E\cap [\beta,\delta)$.
Suppose $\delta\leq \gamma$ is a limit point of $E$, $q\in \mathbb P_\delta$ and $N\in \M_q$.
Then there is an ordinal $\bar{\delta}\leq \delta$ which is also a limit point of $E$ and
$t\in \mathbb P_{<\bar{\delta}}\rest q$ such that $t$ makes $N$ active at $\bar{\delta}$ and
no condition in $\mathbb P_{<\delta}\rest (q\wedge t)$ makes $N$ active at any ordinal $\eta\in  E\cap (\bar{\delta},\delta]$.
\end{cla}

\begin{proof} If $N\cap \ORD$ is contained in $\delta$ then  we let $\bar{\delta}$ be $\sup(N\cap \delta)$.
By Lemma \ref{E_A-cofinal}, $\bar{\delta}$ is also a limit member of $E$. Now, if
$\beta \in E$ and $G_\beta$ is $V$-generic over $\mathbb P_\beta$  then $\sup (N(G_\beta)\cap \ORD)=\bar{\delta}$,
hence  $N$ cannot be made active at any ordinal $\eta \in E \cap (\bar{\delta},\delta]$, and
we can let $t$ be any condition in $\mathbb P_{<\bar{\delta}}\rest q$.
Suppose now that $N$ has some ordinal $\geq \delta$ and let $\delta^*$ be the least ordinal $\geq \delta$
that can  belong to $N(G_\beta)$, for some $\beta \in E\cap \delta$ and some $V$-generic $G_\beta$ over
$\mathbb P_\beta$ such that $q\rest \beta \in G_\beta$.
By taking the least possible $\beta$, we may assume that $\beta \in N(G_{<\beta})$ and hence we can form the model $N[G_\beta]$.
Let $\lambda$ be the cofinality of $\delta^*$ in $N[G_\beta]$.
If  $\lambda < \delta^*$, i.e. $\delta^*$ is singular in $N[G_\beta]$, then $\lambda < \delta$, hence $\lambda \in N[G_\beta]\cap \delta$.
Now,  by increasing $\beta$ and going to a further generic extension, we may assume that $\lambda < \beta$.
Since $\mathbb P_\beta$ collapses all cardinals $<\beta$ to have cardinality at most $\omega_1$,
it follows that we may assume that the cofinality of $\delta^*$ in $N[G_\beta]$
is either $\omega$, $\omega_1$, or $\delta^*$. Suppose $\delta^*$ is regular in $N[G_\beta]$.
Then, since $\delta$ is strong limit and $\delta^*$ is the least ordinal of $N[G_\beta]\setminus \delta$,
it follows that $\delta^*$ is also strong limit and hence inaccessible in $N[G_\beta]$.
% By increasing $\beta$ again we may assume that $\beta \in N(G_\beta)$.
Now, let $\bar{\delta}= \sup (N[G_\beta]\cap \delta)$. By elementarity of $N(G_\beta)$ in ${\rm Hull}(N,V_\delta)$,
$\bar{\delta}$ is a limit point of $E$.
Now, let $t\in G_\beta$ be a condition  that makes $\beta$ attainable from $N$,
forces that  $\delta^*\in N[\Gdot_\beta]$ and decides  whether the cofinality of $\delta^*$ in $N[\Gdot_\beta]$
is $\omega$, $\omega_1$, or $\delta^*$. Since $q\rest \beta \in G_\beta$, we may also assume that $t\leq q\rest \beta$.
We claim that these $\bar{\delta}$ and $t$ work.
It suffices to show the following.

\begin{subclaim*}
$(q\wedge t)\rest \eta \forces_{\mathbb P_\eta} \sup(N(\Gdot_\eta)\cap \delta)= \bar{\delta}$, for all $\eta \in E\cap [\beta,\delta)$
\end{subclaim*}

\begin{proof}
To see this, note that if the cofinality of $\delta^*$ in $N[G_\beta]$ is $\omega$ then $N[G_\beta]$
contains a cofinal $\omega$-sequence in $\delta^*$.
It follows that $\delta^*= \delta$ and $N[G_\beta]\cap \delta$ is cofinal in $\delta$,
hence $\bar{\delta}=\delta$. Therefore, in this case, $t\forces_{\mathbb P_\eta} \sup (N(\Gdot_\eta)\cap \delta)=\delta$,
for all $\eta \in E\cap [\beta,\delta)$.
Suppose now $\delta^*$ is inaccessible in $N[G_\beta]$. Then, for every  $\eta \in N[G_\beta] \cap E\cap \delta$
which is bigger than $\beta$,  the forcing notion $\mathbb P_\eta/G_\beta$ is smaller than $\delta^*$, hence if
$G_\eta$ is $V$-generic over $\mathbb P_\eta$ extending $G_\beta$ then
$\sup(N[G_\eta]\cap \delta^*)= \sup(N[G_\beta]\cap \delta^*)=\bar{\delta}$.
Moreover, $\delta^*$ remains inaccessible in $N[G_\eta]$ and so the same properties holds for $N[G_\eta]$
in place of $N[G_\beta]$.
Finally, suppose $\delta^*$ has cofinality $\omega_1$ in $N[G_\beta]$.
Suppose $\eta \in E\cap \delta$ with $\eta >\beta$ and $G_\eta$ is a $V$-generic filter over $\mathbb P_\eta$
extending $G_\beta$ and such that $(q\wedge t) \rest \eta \in G_\eta$.
By our inductive assumption, we have that $N(G_\eta)\cap \omega_1 = N\cap \omega_1$. Therefore, we also have that
$N(G_\eta)\cap \omega_1 = N(G_\beta)\cap \omega_1$. Since $\delta^*$ has cofinality $\omega_1$ in
$N[G_\beta]$, it follows that we also have $\sup(N(G_\eta)\cap \delta^*)= \sup(N(G_\beta)\cap \delta^*)=\bar{\delta}$.
This completes the proof of the subclaim and Claim \ref{stability}.
\end{proof}
\end{proof}

Back to the proof of Theorem \ref{main-thm} in the case $\gamma$ is a limit point of $E$.
First, note that if $p$ has no extension that makes $\gamma$ attainable from $M$ then
the statement we are trying to prove follows by induction.
So, by strengthening $p$ we may assume that it makes $\gamma$ attainable from $M$.
By applying Claim \ref{stability} for $\delta=\gamma$ and strengthening $p$ again  we may fix $\bar{\gamma}\leq \gamma$
which is a limit of $E$ and $\beta \in E\cap \bar{\gamma}$ such that $p\! \restriction \beta$ makes $M$ active at $\bar{\gamma}$ and
no condition $q\in \mathbb P_{<\gamma}\! \restriction \! p$ makes $M$ active at any ordinal in $E\cap (\bar{\gamma},\gamma]$.
We may also assume that $p\! \restriction \! \beta$ makes $\beta$ and $\gamma$ attainable from $M$.
% and  $p\! \restriction \! \beta \forces_{\mathbb P_\beta} \sup(M[\Gdot_\beta]\cap \bar{\gamma})=\bar{\gamma}$.
%Note that then $p\! \restriction \! \beta$ also forces that $E\cap M[\Gdot_\beta]$ is also cofinal in $\bar{\gamma}$.
We plan to show
\[
p \forces_{\mathbb P_\gamma} M(\Gdot_\gamma) = \bigcup \{ M(\Gdot_\eta): \eta \in E\cap \bar{\gamma}\}.
\]
By the inductive assumption this implies the statement we wish to prove.
% $p\forces_{\mathbb P_\gamma} M(\Gdot_\gamma)\cap \omega_1$ is $(M,\mathbb P_\gamma)$-semigeneric.
Let us pick an arbitrary $V$-generic filter $G_\beta$ over $\mathbb P_\beta$ such that
$p\rest \beta \in G_\beta$ and let us work in $V[G_\beta]$ for a while. Since $p\rest \beta$ makes
$\beta$ and $\gamma$ attainable from $M$ we can form the model $M[G_\beta]$ and we know that $\gamma \in  M[G_\beta]$.
Moreover, since $p\rest \beta$ makes $M$ active at $\bar{\gamma}$ we know that
$E\cap M(G_\beta) \cap \bar{\gamma}$ is cofinal in $\bar{\gamma}$.
Fix a $\mathbb P_\gamma/G_\beta$-name $\tau \in M[G_\beta]$ for an element  of $V$ and let
$D$ be the set of conditions in $\mathbb P_\gamma/G_\beta$ deciding  the value of $\tau$.
Then $D$ is a dense open subset of $\mathbb P_\gamma/G_\beta$ and belongs to $M[G_\beta]$.
It suffices to show that for every $q\in \mathbb P_\gamma/G_\beta$ with $q\leq p$  there is $r\in D$
and  $s\in \mathbb P_\gamma/G_\beta$ such that $s\leq q,r$ and
$s \rest \eta \forces_{\mathbb P_{\eta}/G_\beta} r \in M(\Gdot_{\eta})$,
for some $\eta \in E\cap \bar{\gamma}$.
Then $M(\Gdot_{\eta})$ will be forced by $s$ to contain the value of $\tau$, as decided by $r$.
Consider now one such $q\leq p$ and assume, without loss of generality, that  $q\in D$.
By applying iteratively Claim \ref{stability} for each $N\in \M_q$ and a genericity argument,
we can find an ordinal $\eta \in E\cap M(G_\beta) \cap \bar{\gamma}$ and $t\in \mathbb P_\eta/G_\beta$
such that $t\leq q\rest \eta$ and  such that, for every $N\in \M_q$,
either $t$ makes $N$ active at $\bar{\gamma}$ or no extension of $q\wedge t$ in
$\mathbb P_{\bar{\gamma}}/G_\beta$ makes $N$ active at any ordinal in $E\cap (\eta,\bar{\gamma}]$.
By increasing $\eta$, we may also assume that $\dom(w_q)\cap \bar{\gamma}$ is a subset of $\eta$.
Let $\M=\M_q^{t,\bar{\gamma}}$.
%be the collection of all the $N\rest \gamma$, for  $N\in \M_q$ which are made
%active at $\bar{\gamma}$ by $t$.
Since $q\rest \beta$ already makes $M$ active at $\bar{\gamma}$ we have that $M\rest \bar{\gamma}\in \M$.
%Since $(q\wedge t)\rest \bar{\gamma}$ belongs to $\mathbb P_{\bar{\gamma}}$,
By condition (1) of Definition \ref{main-def} and increasing $\eta$ again if necessary, we may
assume that $(q\wedge t)\rest \eta$ forces $\M$ to be a weak $\in_{\bar{\gamma}}$-chain at $\eta$.
%we also know that, for sufficiently large $\eta\in E\cap \bar{\gamma}$  and any $V$-generic filter $G_\eta$
%extending $G_\beta$ with  $(q\wedge t)\rest \eta \in G_\eta$, if $P,Q\in \M$ and
%$P\cap \omega_1= Q\cap \omega_1$ then $P=Q$ and if $P\cap \omega_1 < Q\cap \omega_1$
%then $P\in_{\bar{\gamma}} Q(G_\eta)$.
% Let $\N$ be the collection of all those $N\in \M$ such that $N\cap \omega_1 < M\cap \omega_1$.
Let $\{ M_0,\ldots, M_{n-1}\}$ be the enumeration of the members of $\M$ given by their intersections with $\omega_1$.
Suppose $M\rest \bar{\gamma}= M_k$.
%Now pick a sufficiently large $\eta$. Since $M(G_\beta) \cap E$ is cofinal in $\bar{\gamma}$ we can make sure
%that $\eta \in M(G_\beta)$.
Now take any $V$-generic filter $G_\eta$ over $\mathbb P_\eta$ extending $G_\beta$
and containing $(q\wedge t)\! \restriction \! \eta$. Note that, since $\eta \in M(G_\beta)$ we can form the model  $M[G_\eta]$.
For each $i<k$ we know that $M_i \in_{\bar{\gamma}} \! \! (M\! \restriction \! \bar{\gamma})(G_{\eta})$, so there is
$M_i^\#\in M(G_{\eta})$ such that $M_i^\# \! \restriction \! \bar{\gamma} = M_i$.
Let $\bar{M}_i$ be $M_i^\# \! \restriction \! \gamma$. Since $\gamma \in M(G_{\eta})$,
we can compute $\bar{M}_i$ inside $M(G_\eta)$, for all $i<k$.
Let $\N= \{ \bar{M}_i: i <k\}$. It follows that $\N\in M(G_\eta)$. Now, note that $\bar{M}_i \rest \bar{\gamma}=M_i$,
for all $i<k$. Moreover, since $q\! \restriction \! \eta \forces_{\mathbb P_\eta} \bar{M}_i \in_\gamma M(\Gdot_\eta)$
and no condition compatible with $q$ can make $M$ active at any ordinal in $E\cap (\bar{\gamma},\gamma]$
then the same must be true for $\bar{M_i}$, for all $i<k$.
Therefore, if we let $\M_{\bar{q}}= \M_q \cup \N$ and $w_{\bar{q}}=w_q$, then $\bar{q}$ is a condition.
In fact, it is equivalent to $q$ and $\bar{q}\rest \bar{\gamma}=q\rest \bar{\gamma}$.
Let $\bar{D}=D\cap G_\eta$. Then $\bar{D}$ is a dense open subset of $\mathbb P_\gamma/G_\eta$
and belongs to $M[G_\eta]$. Now, recall that $M$ is an $\alpha$-model and $\alpha >\gamma$.
Then  $M[G_\eta]$ is an elementary submodel of ${\rm Hull}(M,V_\alpha)[G_\eta]$.
Since $\alpha >\gamma$, we know that $\bar{q}\in {\rm Hull}(M,V_\alpha)[G_\eta]$.
Since $\bar{q}\rest \eta$, we know that $\bar{q}\in G_\eta$.
Now, by elementarity of $M[G_\eta]$ in ${\rm Hull}(M,V_\alpha)[G_\eta]$, we can find $r\in \bar{D}\cap M(G_\eta)$
such that $r\rest \eta \in G_\eta$  and $\N\subseteq \M_r$.

\begin{cla} $q$ and $r$ are compatible.
\end{cla}

\begin{proof}
We need to define a common extension $s$ of $q$ and $r$.
Let us first fix some $u\in G_\eta$ extending $(q \wedge t)\rest \eta$ and $r\rest \eta$ such
that $u\forces_{\mathbb P_\eta} r \in M(\Gdot_\eta)$.
Notice that $\dom(w_q)\cap [\eta,\bar{\gamma})=\emptyset$, while $w_r\in M[G_\eta]$.
We know that $\sup(M[G_\eta]\cap \gamma) = \bar{\gamma}$, so $\dom(w_r)\subseteq \bar{\gamma}$.
We first define a function $w$ on $\dom(w_r)\setminus \eta$.
So, suppose  $\xi \in \dom(w_r) \setminus \eta$ and let $\nu={\rm next}(\xi)$.
Note that $u\forces_{\mathbb P_\eta} \check{w}_r \in M(\Gdot_\eta)$.
So, if we let $q^*=q \wedge u$ we can apply Lemma \ref{extending-w_p} to find a $\mathbb P_{\xi}$-name
$w(\xi)$ for a condition in $\dot{\mathbb Q}_\xi$
such that $q^*\rest \xi$ forces that $w(\xi)$ is stronger than $w_r(\xi)$ and is
$(N[\Gdot_\xi],\dot{\mathbb Q}_\xi)$-semigeneric, for all $N$ in $\dot{\M}_q^\nu$ with
$M\cap \omega_1 \leq N\cap \omega_1$.
Now, let $\M_s=\M_u \cup \M_q \cup \M_r$ and  $w_s= w_u \cup w \cup w_q\rest [\bar{\gamma},\gamma)$.
We claim that $s=(\M_s,w_s)$ belongs to $\mathbb P_\gamma$ and is a common extension
of $q$ and $r$.
We show that $s\rest \xi$ is a condition, by induction on $\xi \in E\cap (\gamma +1)$.
Note that $s\rest \eta =u$ which is a condition in $\mathbb P_\eta$.  This starts the induction.
Suppose first that $\xi \in E\cap (\eta,\bar{\gamma}+1)$ and we have verified that $s\rest \zeta$ is a condition,
for all $\zeta \in E\cap \xi$.
If $\xi={\rm next}(\zeta)$ for some $\zeta \in \dom(w_r)$ we first check condition (2)
of Definition \ref{main-def}. To begin, note that  $r\rest \zeta$  forces that
$w_r(\zeta)$ is $(N[\Gdot_\zeta],\dot{\mathbb Q}_\zeta)$-semigeneric,
for all $N\in \dot{\M}_r^\xi$ and $q^*\rest \zeta$ forces that $w(\zeta)$ extends $w_r(\zeta)$ and
is $(N[\Gdot_\zeta], \dot{\mathbb Q}_\zeta)$-semigeneric, for all $N\in \dot{\M}_q^\xi$
with $M\cap \omega_1 \leq N\cap \omega_1$. Also, $q^*\rest \zeta$ forces that if some $N\in \M_q$
is active at $\xi$ then $N\cong_\xi \bar{N}$, for some $\bar{N}\in \N$.
Since $\N$ is contained in $\M_r$, it follows that $s\rest \zeta$ forces in $\mathbb P_\zeta$
 that $w(\xi)$ is $(N[\Gdot_\zeta], \dot{\mathbb Q}_\zeta)$-semigeneric, for all $N\in \dot{\M}_s^\zeta$.
Let us now check condition (1) of Definition \ref{main-def} for $s$ at $\xi$.
Let $\zeta_0 \in E\cap \xi$ witness condition (1) for $r \! \restriction \! \xi$ and
$\zeta_1 \in E\cap \xi$ witness (1) for $q^*\rest \xi$. Let $\zeta$ be the maximum of  $\zeta_0,\zeta_1$ and $\eta$.
We check that that for every $\rho \in E^*\cap (\zeta,\xi)$ we have that
$s\rest \rho$ forces in  $\mathbb P_\rho$ that $\dot{\M}_s^{\rho,\xi}$ is a weak $\in_\xi$-chain at $\xi$.
To see this pick an arbitrary $z\in \mathbb P_\rho$ with $z\leq s\rest \rho$
such that $z$ decides $\dot{\M}_s^{\rho,\xi}$. %Note that $\M_s^{z,\xi}=\M_q^{z,\xi}\cup \M_r^{z,\xi}$.
Let $\M^*$ be the set of all $N\in \M_q^{z,\xi}$ such that $N\cap \omega_1 <M\cap \omega_1$.
Note that $\M^*\subseteq \N\rest \xi$ and $\N\subseteq \M_r$, hence $\M^*\subseteq \M_r^{z,\xi}$.
Therefore, $\M_s^{z,\xi} = \M_r^{z,\xi}\cup (\M_q^{z,\xi}\setminus \M^*)$.
On the other hand, $u$ forces in $\mathbb P_\eta$ that $\M_r \in M(\Gdot_\eta)$, hence
it forces that  $N\rest \xi \in_\xi M(\Gdot_\eta)\rest \xi$, for all $N\in \M_r$.
Finally, $z$ forces that $\M_q^{z,\xi}\setminus \M^*$ is a a weak $\in_\xi$-chain
at $\rho$, since it is compatible with $q$ which is a condition.
Now, suppose $\xi \in E\cap (\bar{\gamma},\gamma]$. Since $u\forces_{\mathbb P_\eta} r\in M(\Gdot_\eta)$,
and no condition compatible with $q^*$ can make $M$ active at any ordinal in $E\cap (\bar{\gamma},\gamma]$,
then no condition compatible with $q^*$ can make any $N\in\M_r$ active at any ordinal
in $E\cap (\bar{\gamma},\gamma]$.
Therefore, the fact that $q^*$ satisfies Definition \ref{main-def} implies
that so does $s$ at ordinals in $E\cap (\bar{\gamma},\gamma]$.
This completes the verification that $s\in \mathbb P_\gamma$.
Clearly, $s$ is then a common extension of $q$ and $r$,
as required.
\end{proof}

This completes the proof of Theorem \ref{main-thm}.
\end{proof}

\begin{defi}\label{star-def} For $\alpha \in E^*$, let $\mathbb P^*_{\alpha}$ be
the set of all conditions $p\in \mathbb P_\alpha$ such that there are no $\gamma$-models
in $\M_p$, for any $\gamma$ which is either inaccessible or has cofinality $\omega_1$.
\end{defi}

\begin{lemma}\label{density} ${\mathbb P}^*_\alpha$ and $\mathbb P_\alpha$
are equivalent forcing notions, for all $\alpha \in E^*$.
\end{lemma}

\begin{proof} We may assume that $\alpha\in E$. Let $\bar{\mathbb P}_\alpha$
be the set of all conditions $p\in \mathbb P_\alpha$ such that
there is $\beta \in E$ with $\beta >\alpha$ and a $\beta$-model $M_p$ such
that $M_p\rest \alpha \in \M_p$, and $\alpha$ and $\M_p \setminus \{ M_p\rest \alpha\}$ belong to $M_p$.
Note that, by Lemma \ref{adding-M}, $\bar{\mathbb P}_\alpha$ is a dense subset of $\mathbb P_\alpha$.
Now, for each $\gamma\in E \cap (\alpha+1)\cap M_p$ which is either inaccessible or has cofinality $\omega_1$,
let $\gamma_p=\sup(\gamma \cap M_p)$. Then $\gamma_p$ has countable cofinality and, by Lemma \ref{E_A-cofinal},
we know that $\gamma_p \in E$.
Note that, for every such $\gamma$, the model $M_p\rest \alpha$ cannot be  made active at
any ordinal $\eta \in E\cap (\gamma_p,\gamma]$ by a condition  compatible with $p$.
For $\gamma$ inaccessible this follows from the fact that $\mathbb P_\xi$ has size $<{\rm next}(\xi)$, for
all $\xi \in E$. For $\gamma$ of cofinality $\omega_1$ this follows from Theorem \ref{main-thm}.
%implies that this holds for $\gamma$ of cofinality $\omega_1$.
Since all the models in $\M_p$ other than $M_p \rest \alpha$ are elements of $M_p$,
it follows that no model in $\M_p$ can be made active at any such ordinal by any condition
compatible with $p$. Now define a map $\pi_\alpha:\bar{\mathbb P}_\alpha \rightarrow \mathbb P^*_\alpha$
by letting $\pi_{\alpha}(p)=(\M_p^*,w_p)$, where $\M_p^*$ is obtained by replacing
any $\gamma$-model $N\in \M_p$, for $\gamma$ either inaccessible or of cofinality $\omega_1$,
by $N\rest \gamma_p$. By the above argument $\pi_{\alpha}(p)$ and $p$ are equivalent
conditions in $\mathbb P_\alpha$, i.e. any condition $q \in \mathbb P_\alpha$ is compatible
with $\pi_\alpha(p)$ iff it is compatible with $p$. It follows that $\mathbb P_\alpha$
and $\mathbb P^*_\alpha$ are equivalent as forcing notions.
\end{proof}

Now, note that if $\alpha \in E$ is either inaccessible or has cofinality $\omega_1$
then $\mathbb P^*_\alpha$ is the direct limit of the $\mathbb P^*_\xi$, for $\xi \in E\cap \alpha$.
Therefore, if we let $\mathbb P^*_\kappa$ be the union of the $\mathbb P^*_\alpha$, for $\alpha \in E$,
then by a standard argument $\mathbb P^*_\kappa$ satisfies the $\kappa$-chain condition.
By putting all this together we get the final conclusion.

\begin{teo}\label{final} $\mathbb P^*_\kappa$ is semiproper and satisfies the $\kappa$-chain condition.
\qed
\end{teo}

\bibliographystyle{abbrv}
\bibliography{side}

\end{document}